\documentclass[11pt]{amsart}
\usepackage{amssymb, latexsym, graphicx, mathrsfs, enumerate, amsmath, amsthm, textcomp, url, tikz-cd,bbm}
\usepackage{stmaryrd}
\usepackage{mathdots}
\usepackage[T1]{fontenc}
\usepackage{comment}

\input xy
\xyoption{all}

\allowdisplaybreaks
\allowdisplaybreaks

\newtheorem{theorem}{Theorem}
\newtheorem{proposition}[theorem]{Proposition}

\theoremstyle{definition}

\newcommand{\mfo}{\mathfrak{o}}

\newcommand*{\Rom}[1]{\expandafter\@slowromancap\romannumeral #1@}
\begin{document}

\title{On a Kostant Section for the unitary group}


\author[Yuchan Lee]{Yuchan Lee}
\thanks{The author is supported by  Samsung Science and Technology Foundation under Project Number SSTF-BA2001-04.}

\address{Yuchan Lee \\  Department of Mathematics, POSTECH, 77, Cheongam-ro, Nam-gu, Pohang-si, Gyeongsangbuk-do, 37673, KOREA}

\email{yuchanlee329@gmail.com}

\begin{abstract}
For the unitary group defined over the ring of integers in a non-Archimedean local field, we give a correction for a Kostant section provided in G. Laumon and B.C. Ng\^o's paper; Le lemme fondamental pour les groupes unitaires. 
\end{abstract}

\maketitle
\section{Introduction}

Let $F$ be a non-Archimedian local field and $\mfo$ be its ring of integers with $\kappa$ its residue field.
Let $\mathrm{G}$ be a connected reductive group over $\mfo$ with Lie algebra $\mathfrak{g}$.
For the geometric invariant theory (GIT) quotient $\mathfrak{g}\sslash\mathrm{G}$ of $\mathfrak{g}$, we consider the adjoint quotient map
 \[\varphi_n: \mathfrak{g}\rightarrow \mathfrak{g}\sslash \mathrm{G}.\]
Kostant, in \cite{Ko}, provided a section to $\varphi_n$ over an algebraically closed field of characteristic 0. The section is named by a Kostant section for $\mathrm{G}$.
In \cite[Section 2.3]{AFV},  they presented a Kostant section over a ring $R$ for a connected split reductive group and they outlined the assumption on 
$R$ for each case in Remark 8 in loc.cit.
On the other hand, for a quasi-split reductive group, A. Bouthier and K. Česnavičius  proved the existence of a Kostant section under the restrictions on the characteristic of the residue field in \cite[Proposition 4.3.2]{BV}.

In the case of the unitary Lie algebra $\mathfrak{u}_n$, the quotient map $\varphi_n$ is precisely a morphism which maps an element in $\mathfrak{u}_n$ to its coefficients of the characteristic polynomial.
The result in \cite[Proposition 4.3.2]{BV} ensures the existence of a Kostant section for the unitary group if $\mathrm{char}(\kappa)$ does not divide the rank $n$ of the unitary group.
On the other hand, if $\mathrm{char}(\kappa)\neq 2$, then the existence of  a Kostant section for the unitary group over $\mathfrak{o}$ is derived in  \cite[Section 2.3]{LN}.
More precisely, 
they provide the description of a Kostant section over $\mfo$, as a matrix in page 489 of loc.cit.
However, we have found that this matrix is not contained in the unitary Lie algebra.

In this manuscript, we give  a correction to the construction of \cite[Section 2.3]{LN} in Proposition \ref{Ngo_Laumon_correction}. Combining it with the result of \cite{BV},  we will state a general result about the existence of a Kostant section for $\mathfrak{u}_n$ in Theorem \ref{main_thm}. 
This will be used in my joint work in \cite{CKL}.
\\

\textbf{Acknowledgments.} The author sincerely thanks Alexis Bouthier, Kęstutis Česnavičius, Sungmun Cho, and Taeyeoup Kang to provide helpful comments and encouragement. 
\section{Notation}
\begin{itemize}
    \item 
Let $F$ be a non-Archimedean local field with its ring of integers $\mfo$ and a uniformizer $\pi$. Let $\kappa$ be the residue field of $\mfo$ and $\bar{\kappa}$ be the algebraically closure of $\kappa$.

    \item
Let $E$ be a quadratic unramified field extension over $F$ with its ring of integers $\mfo_E$ and let $\sigma$ be the non-trivial element in $\mathrm{Gal}(E/F)$.
 \item
 Let $L$ be a free $\mfo_E$-module of rank $n$. We fix a hermitian form $h:L\times L \rightarrow \mfo_E$ whose Gram matrix is given by
 \begin{equation}\label{hermitian}
 \Phi_n:=
 \begin{pmatrix}
     0&\cdots&1\\
     \vdots&\iddots&\vdots\\
     1&\cdots&0
 \end{pmatrix}.
 \end{equation}
 \item 
For unitary Lie algebra $\mathfrak{u}_n$, we define the following map over $\mfo$ by
    \[
    \varphi_n:\mathfrak{u}_{n}(\mfo)\rightarrow
    \bigoplus_{i=1}^n\{a\in\mfo_E\mid \sigma(a)=(-1)^i a\},
    \gamma \mapsto (\textit{coefficient of $\chi_\gamma$})
    \]
    where $\chi_\gamma$ is the characteristic polynomial of $\gamma$.
\end{itemize}

\section{Kostant section for $\mathfrak{u}_n$}
    
The existence of the Kostant section for $\varphi_n$ is explained in \cite[Section 2.3]{LN} when the characteristic of $F$ is not $2$.
However, the matrix in page 489 of loc. cit. is not an element of the unitary Lie algebra, whereas its characteristic polynomial is correctly stated.
The following matrix suggests a correction to  \cite[Section 2.3]{LN}. 
\begin{proposition}\label{Ngo_Laumon_correction}
    If $\mathrm{char}(\kappa)\neq 2$, then the following matrix $X$ yields a Kostant section over $\mfo$; 
    \[
    X=\begin{pmatrix}
        -b_1 &-\alpha^{-1}b_2&\cdots&\cdots&-\alpha^{-(n-2)}b_{n-1}&-2\alpha^{-(n-1)}b_n\\
        \alpha&0&\cdots&\cdots &0&-\alpha^{-(n-2)}b_{n-1}\\
        0&\ddots&\ddots&&\vdots&\vdots\\
        \vdots&\ddots&\ddots&\ddots&\vdots&\vdots\\
        \vdots&&\ddots&\ddots&0&-\alpha^{-1}b_2\\
        0&\cdots&\cdots&0&\alpha&-b_1
    \end{pmatrix}
    \]
where $\alpha\in \mfo_E^{\times}$ such that $\alpha+\sigma(\alpha)=0$ and $b_i$ are coefficients defined in \cite[Section 2.3]{LN}.
\end{proposition}
\begin{proof}
First at all, the existence of such an element $\alpha \in \mfo_{E}^{\times}$ is verified as follows.
Since the $F$-linear map $\sigma$ is the non-trivial involution, there does exist an eigenvector with eigenvalue $-1$.
Then multiplying this eigenvector by a suitable power of $\pi$, we obtain a desired element $\alpha\in \mfo_E^{\times}$.

We fix a polynomial $x^{n}+a_1x^{n-1}+\cdots+a_n\in \mfo_E[x]$ such that $\sigma(a_i)=(-1)^ia_i$.  We consider the diagonal matrix \[D_{\alpha,n}:=\begin{pmatrix}
    1&0&\cdots&0\\
    0&\alpha&\cdots&0\\
    \vdots&\vdots&\ddots&\vdots\\
    0&0&\cdots&\alpha^{n-1}
\end{pmatrix}\text{. We then have }D_{\alpha,n}^{-1} XD_{\alpha,n}
=\begin{pmatrix}
        -b_1 &-b_2&\cdots&\cdots&-b_{n-1}&-2b_n\\
        1&0&\cdots&\cdots &0&-b_{n-1}\\
        0&\ddots&\ddots&&\vdots&\vdots\\
        \vdots&\ddots&\ddots&\ddots&\vdots&\vdots\\
        \vdots&&\ddots&\ddots&0&-b_2\\
        0&\cdots&\cdots&0&1&-b_1
    \end{pmatrix}.
\] The characteristic polynomial of $D_{\alpha,n}^{-1}XD_{\alpha,n}$  is the same as that of $X$ which is $x^{n}+a_1x^{n-1}+\cdots+a_n$ according to \cite[Section 2.3]{LN}.

    Therefore it suffices to show that the matrix $X$ is contained in $\mathfrak{u}_n(\mfo)$.
    By the construction of $b_i$ in \cite[Section 2.3]{LN}, we have that $\sigma(b_i)=(-1)^i b_i$ and that $\sigma(\alpha^{i-1}b_i)=(-1)^{i-1}\cdot(-1)^i\cdot\alpha^{i-1} b_i=-\alpha^{i-1} b_i$.
    On the other hand, we have \[(\Phi_n X +\sigma({}^tX)\Phi_n)_{i,j}=X_{n+1-i,j}+\sigma(X_{n+1-j,i}).\]
    Here, $\Phi_n$ is given in Equation (\ref{hermitian}).
    By putting $i'=n+1-i$, it suffices to verify that
    \[
    X_{i',j}+\sigma(X_{n+1-j,n+1-i'})=0,
    \]for $1\leq i',j\leq n$.
    We note that $X_{i,j}=\left\{\begin{array}{c l}
       -\alpha^{-(k-1)}b_k  & (i,j)= (1,k) \textit{ or }(n+1-k,n) \textit{ for $1\leq k\leq n-1$}; \\
        -2\alpha^{-(n-1)}b_n &  (i,j)=(1,n); \\
         \alpha&(i,j)=(k+1,k) \textit{ for $1\leq k\leq n-1$};\\
         0 & \textit{otherwise}.
    \end{array}\right.$
\begin{enumerate}
    \item In the case that $(i',j)= (1,k) \textit{ or }(n+1-k,n) \textit{ for $1\leq k\leq n-1$}$, we have
    \begin{align*}
        &X_{1,k}+\sigma(X_{n+1-k,n})=-\alpha^{-(k-1)}b_k+\sigma(-\alpha^{-(k-1)}b_k)=0\textit{ or};\\
        &X_{n+1-k,n}+\sigma(X_{1,k})=-\alpha^{-(k-1)}b_k+\sigma(-\alpha^{-(k-1)}b_k)=0.
    \end{align*}
    \item In the case that $(i',j)=(1,n)$, we have
    \[
    X_{1,n}+\sigma(X_{1,n})=-2\alpha^{-(n-1)}b_n+\sigma(-2\alpha^{-(n-1)}b_n)=0.
    \]
    \item In the case that $(i',j)=(k+1,k)$\textit{ for $1\leq k\leq n-1$}, we have
    \[
    X_{k+1,k}+\sigma(X_{n-k+1,n-k})=\alpha+\sigma(\alpha)=0.
    \]
    \item Otherwise, by the above computation, we can deduce that if $X_{i',j}=0$, then $X_{n+1-j,n+1-i'}=0$.
\end{enumerate}
Therefore, we conclude that $X\in \mathfrak{u}_n(\mfo)$ if $\mathrm{char}(\kappa)\neq 2$.
\end{proof}

\begin{theorem}\label{main_thm}
    The morphism $\varphi_n$ admits a Kostant section $\left\{\begin{array}{l l}
         \textit{over $\mfo$} &\textit{if $n$ is odd};\\
         \textit{over $\mfo$ with $\mathrm{char}(\kappa)\neq 2$} & \textit{if $n$ is even}.
    \end{array}\right.$
\end{theorem}
\begin{proof}
According to \cite{BV}[Proposition 4.3.2], for $\mathfrak{u}_n$, the morphism $\varphi_n$ admits a Kostant section if $\mathrm{char}(\kappa)$ divides neither $\#\pi_1(((\mathrm{U}_{n,\bar{\kappa}})_{\mathrm{der}})^{\mathrm{ad}})$ nor any coefficient in the expression of a root of $\mathrm{U}_{n,\bar{\kappa}}$ in terms of a base of simple roots.
Since $\mathrm{U}_{n,\bar{\kappa}}\cong \mathrm{GL}_{n,\bar{\kappa}}$, we can deduce that the coefficients in the second condition is $1$ or $-1$.
It then suffices to consider $\#\pi_1(((\mathrm{U}_{n,\bar{\kappa}})_{\mathrm{der}})^{\mathrm{ad}})$. Since $\mathrm{SL}_{n,\bar{\kappa}}$ is the derived subgroup of $\mathrm{GL}_{n,\bar{\kappa}}$ and it is simply connected, we have
\[
\#\pi_1(((\mathrm{U}_{n,\bar{\kappa}})_{\mathrm{der}})^{\mathrm{ad}})=n.
\]
By combining the result in Proposition \ref{Ngo_Laumon_correction}, we conclude that the morphism $\varphi_n$ admits a Kostant section over any $\mfo$ in the case that $n$ is odd. 
\end{proof}


\begin{thebibliography}{99}
\bibitem[AFV]{AFV}
J.D.Adler, J.Fintzen, and S.Varma, \textit{On Kostant sections and topological nilpotence}, J. Lond. Math. Soc. \textbf{97} (2018), no. 2, 325--351.
\bibitem[BV]{BV}
A.Bouthier and K.\v{C}esnavi\v{c}ius, \textit{Torsors on loop groups and the {H}itchin fibration}, Ann. Sci. \'{E}c. Norm. Sup\'{e}r. \textbf{55} (2022), no.3, 791--864.
\bibitem[CKL]{CKL}
S.Cho, T.Kang, and Y.Lee, \textit{Stable orbital integrals for classical Lie algebras and smooth integral models}, preprint (2023), 139 pp.
\bibitem[Ko]{Ko}
B.Kostant. \textit{Lie group representations on polynomial rings}, Amer. J. Math. \textit{85} (1963), 327--404.
\bibitem[LN]{LN}
G.Laumon and B.C.Ng\^{o}, \textit{Le lemme fondamental pour les groupes unitaires}, Ann. of Math. (2) \textit{168} (2008), no.2, 477--573.
\end{thebibliography}
\end{document}